\documentclass{article}%

\usepackage{amsmath,amssymb,amsfonts}%
\usepackage{theorem}%
\usepackage{color}%
\usepackage{hyperref}%
\usepackage{bbm}%

\theoremstyle{change}%
\newtheorem{definition}{Definition:}[section]%
\newtheorem{proposition}[definition]{Proposition:}%
\newtheorem{theorem}[definition]{Theorem:}%
\newtheorem{lemma}[definition]{Lemma:}%
{\theorembodyfont{\rmfamily}\newtheorem{remark}[definition]{Remark:}}%
{\theorembodyfont{\rmfamily}\newtheorem{example}[definition]{Example:}}%

\newenvironment{proof}
  {{\bf Proof:}}
  {\qquad \hspace*{\fill} $\Box$}%

\newcommand{\fg}{\mathfrak{g}}%
\newcommand{\fl}{\mathfrak{l}}%
\newcommand{\fn}{\mathfrak{n}}%
\newcommand{\fs}{\mathfrak{s}}%
\newcommand{\fh}{\mathfrak{h}}%
\newcommand{\fr}{\mathfrak{r}}%
\newcommand{\fe}{\mathfrak{e}}%
\newcommand{\fo}{\mathfrak{o}}%

\newcommand{\ad}{\operatorname{ad}}%

\newcommand{\inner}{\operatorname{int}}%
\newcommand{\cl}{\operatorname{cl}}%
\newcommand{\rme}{\mathrm{e}}%

\newcommand{\CC}{\mathcal{C}}%
\newcommand{\SC}{\mathcal{S}}%
\newcommand{\UC}{\mathcal{U}}%
\newcommand{\DC}{\mathcal{D}}%
\newcommand{\XC}{\mathcal{X}}%
\newcommand{\AC}{\mathcal{A}}%
\newcommand{\N}{\mathbb{N}}%
\newcommand{\R}{\mathbb{R}}%

\begin{document}

\title{Controllability of linear systems on Lie groups with finite semisimple center}
\author{Adriano Da Silva\footnote{Supported by Fapesp grant $n^o$ 2013/19756-8}\\
Instituto de Matem\'atica,\\
Universidade Estadual de Campinas\\
Cx. Postal 6065, 13.081-970 Campinas-SP, Brasil\\
and\\
V\'ictor Ayala\footnote{Supported by Proyecto Fondecyt $n^o$ 1150292, Conicyt}\\
Universidad de Tarapac\'a\\
Instituto de Alta Investigaci\'on\\
Casilla 7D, Arica, Chile.
}
\date{\today}
\maketitle

{\bf Abstract.} A linear system on a connected Lie group $G$ with Lie algebra $\fg$ is determined by the family of differential equations
$$\dot{g}(t)=\XC(g(t))+\sum_{j=1}^m u_j(t)X^j(g(t)),$$
where the drift vector field $\XC$ is a linear vector field induced by a $\fg$-derivation $\DC$, the vector fields $X^j$ are right invariant and $u\in\UC\subset L^{\infty}(\R,\Omega\subset\R^m)$ with $0\in\inner\Omega$. Assume that any semisimple Lie subgroup of $G$ has finite center and $e\in\inner\AC_{\tau_0}$, for some $\tau_0>0$. Then, we prove that the system is controllable if the Lyapunov spectrum of $\DC$ reduces to zero. The same sufficient algebraic controllability conditions were applied with success when $G$ is a solvable Lie group, \cite{DaSilva1}.

\bigskip
{\bf Key words.} Lie group, derivation, linear system, controllability.

\bigskip
{\bf 2010 Mathematics Subject Classification.} 16W25, 93B05, 93C05.

\section{Introduction}

Throughout the paper $G$ will stands for a connected Lie group with Lie algebra $\fg$. In \cite{VAJT} the authors introduced the class of linear system on $G$, determined by the family of differential equations
$$\dot{g}(t)=\XC(g(t))+\sum_{j=1}^m u_j(t)X^j(g(t)),$$
where the drift $\XC$ with flow $(\varphi_t)_{t\in\R}$, is a linear vector field induced by a $\fg$-derivation $\DC$ through the formula
$$\varphi_t(\exp Y)=\exp(\rme^{\DC}Y), \mbox{ for all } t \in \R,Y \in \fg.$$

The vector fields $X^j$ are right invariant and $u\in\UC \subset L^{\infty}(\R, \Omega\subset \R^m)$ is the class of admissible controls. Additionaly, if $\Omega$ is a compact, convex subset such that $0\in\inner\Omega$, the system is called {\it restricted}. Otherwise the systems is called {\it unrestricted}.

In \cite{VAJT} it is proved that the {\it ad-rank} condition
$$\mathrm{Span}\{X^j,\ad^k(X^j): j = 1, \ldots ,m \mbox{ and } k\geq 1\}=\fg$$
implies local controllability of the linear system from the identity element $e\in G$. In particular $e\in\inner\AC_{\tau}$ for any $\tau>0$. On the other hand, in \cite{VASM} the authors give a first example of a local controllable linear system from $e$ where the reachable set $\AC(e) = \AC$ is not a semigroup. Later in \cite{JPh} Jouan proves that $\AC$ is a semigroup if and only if $\AC = G$. All this facts show that it is hard to understand the class of linear systems from the controllability point of view. For example, controllability of invariant control system, i.e., when the drift below is also an invariant vector field, is a local property for connected groups. Recently, in \cite{DaSilva1} Da Silva shows that controllability property for linear systems on solvable Lie groups is obtained if $\AC$ is open and if any eigenvalue of the derivation $\DC$ has zero real part. Moreover, controllability of restricted systems on nilpotent Lie groups are a very exceptional issue since the above conditions are also necessary. Which means that the derivation has to be very unstable.

Our aim is to understand the controllability behavior of linear systems. In particular, to extend the results in $\cite{DaSilva1}$ as much as possible. It turns out that if the Lie group $G$ has finite semisimple center (see Definition \ref{finitess}) such extension is possible. Next, we state our main

\begin{theorem}
Let $G$ be a connected Lie group with finite semisimple center. Therefore, a linear system on $G$ is controllable if $e\in\inner\AC_{\tau_0}$, for some $\tau_0>0$, and $\DC$ has just eigenvalues with zero real part.
\end{theorem}

Furthermore, we introduce a special class of linear system which could be rele\-vant for applications. It is a linear system on a direct product of Lie groups. More precisely, linear on the first component and homogeneous on the second one. Through our main result we show that in positive time it is possible to connect any two configuration of a rolling sphere S over a revolving plane around the $z$-axis with constant angular velocity. The kinematic equations and the configuration space state $G=\R^2\times SO(3)$ appears in the book of Velimir Jurdjevic, \cite{VJ}.

As we mention, controllability property of linear systems is very unusual. Regarding that, it is neccessary to have a more convenient approach to this problem by study the existence of control sets and the possible number of these objects. We will submit a paper in this direction soon. 

The paper is organized as follows. Section 2 contains the definition of linear vector fields and linear systems on Lie groups. It also shows the decomposition of $\fg$ through the generalized eigenspaces of $\DC$ and the subgroups of $G$ associated with them. Finally, some basic properties of reachable sets of restricted linear systems are given. In Section 3 we prove the main results of the article which strongly depend on the fact that the Lie subgroup $G^0=\langle\exp(\fg^0)\rangle$ is contained in the reachable set $\AC$ when the group $G$ has finite semisimple center. Here $\fg^0$ is the sum of all generalized eigenspaces of $\DC$ for eigenvalues with zero real part. Section 4 contains a special class of linear systems which we believe could be relevant for concrete applications.

\section{Preliminares}

This section contains precise definitions and basic results concerning a linear vector field and linear systems. For more on the subjects the reader may consult \cite{VASM}, \cite{VAJT}, \cite{DaSilva1}, \cite{JPh}, \cite{JPhDM} and \cite{JPh1}.

\subsection{Linear vector fields and $\DC$-decomposition}

Let $G$ be a connected Lie group of dimension $d$ with Lie algebra $\fg$. The {\bf normalizer} $\eta$ of $\fg$, introduced in \cite{VAJT} in the control system set up, is defined by
$$\eta=\mathrm{norm}_{X^{\infty}(G)}(\fg)=\{F\in X^{\infty}(G); \,\mbox{ for all }Y\in\fg, \;\;[F, Y]\in\fg\}$$
where $X^{\infty}(G)$ denotes the Lie algebra of the $\CC^{\infty}$ vector fields on $G$. When $G$ is simply connected the authors proved that $\eta\simeq\fg\otimes_s\partial(\fg)$ and in the connected case, $\eta\simeq\fg\otimes_s\mathrm{aut}(\fg)$. Here $\partial(\fg)$ stands for the Lie algebra of $\fg$-derivations, $\mathrm{aut}(G)\subset\partial(\fg)$ denotes the Lie algebra of $\mathrm{Aut}(G)$, the group of $G$-automorphims, and $\otimes_s$ the semidirect product. The relevance of the normalizer comes from the fact that $\eta$ contains the complete dynamic of any linear system for the class of piecewise constant admisible control. In fact, for any constant control $u=(u_1, \ldots, u_m)\in\Omega$ the vector field $\XC+\sum_{j=1}^m u_jX^j$ belongs to the normalizer (See Definition 4). Precisely,

\begin{definition}
A vector field $\XC$ on $G$ is called {\it linear} if it belongs to $\eta$ and if $\XC(e)=0$.
\end{definition} 

In \cite{JPh} the author gives equivalent conditions for a vector field on $G$ to be linear.

\begin{theorem}
\label{propertieslinear}
Let $\XC$ be a vector field on a connected Lie group $G$. The fo\-llowing conditions are equivalent:
\begin{itemize}
\item[1.] $\XC$ is linear;
\item[2.] The flow of $\XC$ is a one parameter group of automorphisms of $G$;
\item[3.] $\XC$ satisfies
$$\XC(gh)=(dL_g)_h\XC(h)+(dR_h)_g\XC(g), \;\;\mbox{ for all }\;\;g, h\in G, $$
where as usual $L_g$ and $R_h$ denotes the left and right translation by $g$ and $h$, respectively, and $d$ its derivative.
\end{itemize}
\end{theorem}

\bigskip
The last condition shows a kind of ``linear property" of $\XC$. This fact comes also from the linear case on Euclidean space. Just observe that the flow determined by any matrix $A$ of order $d$ satisfy $\rme^{tA} \in \mathrm{Aut}(\R^d)$.

Denotes by $(\varphi_t)_{t\in\mathbb{R}}$ the 1-parameter group of automorphisms of $G$ generated by $\XC$. In particular, $\XC$ is complete and induces the derivation $\DC$ of $\fg$ deﬁned by
$$\mathcal{D}Y=-[\mathcal{X}, Y](e), \mbox{ for all }Y\in\mathfrak{g}.$$
The relation between $\varphi_t$ and $\DC$ is given by the well known formula
\begin{equation}
\label{derivativeonorigin}
(d\varphi_t)_e=\mathrm{e}^{t\mathcal{D}} \;\;\;\mbox{ for all }\;\;\; t\in\mathbb{R}
\end{equation}
which implies in particular that
$$\varphi_t(\exp Y)=\exp(\mathrm{e}^{t\mathcal{D}}Y), \mbox{ for all }t\in\mathbb{R}, Y\in\mathfrak{g}.$$
Conversely, given an element $l\in\mathrm{aut}(G)$ there exists a linear vector field $\XC=\XC(l)$ on $G$. Furthermore, in the simply connected case any $\fg$-derivation induces a linear vector field through the formula of the flow. In fact, $\varphi_t$ is a $G$-automorphism and since $G$ is connected any element of $G$ is a finite product of exponentials.

Let $\alpha$ be an eigenvalue of the derivation $\DC$ and consider the associated $\alpha$-generalized eigenspace $\DC$ given by 
$$\mathfrak{g}_{\alpha}=\{X\in\mathfrak{g}; (\mathcal{D}-\alpha)^nX=0 \mbox{ for some }n=n(\alpha)\geq 1\}.$$
It turns out that $[\mathfrak{g}_{\alpha}, \mathfrak{g}_{\beta}]\subset\mathfrak{g}_{\alpha+\beta}$ when $(\alpha+\beta)\in\mathrm{Spec}(\DC)$ and zero otherwise. This fact allows us to decompose $\fg$ as
$$\fg=\fg^+\oplus\fg^-\oplus\fg^0$$
where
$$\fg^+=\bigoplus_{\alpha\;: \,\mathrm{Re}(\alpha)> 0}\fg_{\alpha},\hspace{.8cm}\fg^-=\bigoplus_{\alpha\;: \,\mathrm{Re}(\alpha)< 0}\fg_{\alpha}\hspace{.8cm} \mbox{ and }\hspace{.8cm}\fg^0=\bigoplus_{\alpha\;: \,\mathrm{Re}(\alpha)=0}\fg_{\alpha}.$$
The subspaces $\fg^{\pm}, \fg^0$ are Lie algebras with $\fg^{\pm}$ nilpotent. We notice that the generalized kernel $\fg_0$ is itself a Lie subalgebra of $\fg$. 

\subsection{Linear systems on Lie groups}

Let $\Omega$ be a subset of $\R^m$ such that $0\in\inner\Omega$ and consider the class of admissible control functions $\UC\subset L^{\infty}(\R, \Omega\subset\R^m)$. 

\begin{definition}
A linear system on a Lie group $G$ is determined by the family of differential equations
\begin{equation}
\label{linearsystem}
\dot{g}(t)=\XC(g(t))+\sum_{j=1}^mu_j(t)X^j(g(t)),
\end{equation}
where the drift vector field $\XC$ is a linear vector field, $X^j$ are right invariant vector fields and $u=(u_1, \cdots,u_m)\in\UC$. 
\end{definition}

For $g\in G$, $u\in\mathcal{U}$ and $t\in\mathbb{R}$ the solution of (\ref{linearsystem}) starting at $g$ is given by
$$\phi_{t, u}(g)=L_{\phi_{t, u}}(\varphi_t(g))=\phi_{t, u}\varphi_t(g),$$
where $\phi_{t, u}:=\phi_{t, u}(e)$ is the solution of (\ref{linearsystem}) starting at the identity element $e\in G$ (see for instance Proposition 3.3 of \cite{DaSilva}). 
The map 
$$\phi:\mathbb{R}\times G\times\mathcal{U}\rightarrow G, \;\;\;(t, g, u)\mapsto\phi_{t, u}(g),$$
satisfies the {\it cocycle property}
$$\phi_{t+s, u}(g)=\phi_{t, \Theta_su}(\phi_{s, u}(g)))$$
for all $t, s\in\mathbb{R}$, $g\in G$, $u\in\mathcal{U}$, where $\Theta_t:\UC\rightarrow\UC$ is the shift flow $u\in\UC\mapsto\Theta_tu:=u(\cdot +t)$. It follows directly from the cocycle property that the diffeomorphism $\phi_{t, u}$ has inverse $\phi_{-t, \Theta_{t}u}$ for any $t\in\R$ and $u\in\UC$. Since for any $t>0$ we have that $\phi_{t, u}(g)$ depends only on $u|_{[0, t]}$ we obtain 
$$\phi_{t, u_2}(\phi_{s, u_1}(g))=\phi_{t+s, u}(g).$$
Here the control $u\in\UC$ is defined by concatenation between $u_1$ and $u_2$ as follows
$$u(\tau)=\left\{
\begin{array}{cc}
u_1(\tau) & \mbox{ for }\tau\in[0, s]\\
u_2(\tau-s) & \mbox{ for }\tau\in[s, t+s].
\end{array}\right.
$$

For any $g\in G$ the sets
\begin{equation}
\label{reachablesets}
\AC_{\tau}(g):=\{\phi_{\tau, u}(g), u\in\UC\}\;\;\;\mbox{ and }\;\;\;\AC(g):=\bigcup_{\tau>0}\AC_{\tau}(g),
\end{equation}
are the {\it set of reachable points from $g$ at time $\tau$}  and the {\it reachable set of $g$}, respectively. When $g=e$ is the identity element of $G$ the sets $\AC_{\tau}(e)$ and $\AC(e)$ are said to be the {\it reachable set at time $\tau$} and the {\it reachable set} and are denoted by $\AC_{\tau}$ and $\AC$ respectively.

\begin{definition}
The linear system (\ref{linearsystem}) on $G$ is said to be controllable if given  $g, h\in G$ there are $u\in\UC$ and $\tau>0$ such that $\phi_{\tau, u}(g)=h$.
\end{definition}

\begin{proposition}
\label{reachable}
It holds:
\begin{itemize}
\item[1.] if $0\leq \tau_1\leq \tau_2$ then $\AC_{\tau_1}\subset\AC_{\tau_2}$
\item[2.] for all $g\in G$, $\AC_{\tau}(g)=\AC_{\tau}\varphi_{\tau}(g)$;
\item[3.] for all $\tau, \tau'\geq 0$ we have $\AC_{\tau+\tau'}=\AC_{\tau}\varphi_{\tau}(\AC_{\tau'})=\AC_{\tau'}\varphi_{\tau'}(\AC_{\tau})$ and inductively that
$$\AC_{\tau_1}\varphi_{\tau_1}(\AC_{\tau_2})\varphi_{\tau_1+\tau_2}(\AC_{\tau_3})\cdots\varphi_{\sum^{n-1}_{i=1}\tau_i}(\AC_{\tau_n})= \AC_{\sum^n_{i=1}\tau_i}$$ 
for any positive real numbers $\tau_1, \ldots, \tau_n$;
\item[4.] for all $u\in\UC$, $g\in G$ and $t>0$ we have that $\phi_{t, u}(\AC(g))\subset\AC(g)$;
\item[5.] $e\in\inner\AC$ if and only if $\AC$ is open.
\end{itemize}
\end{proposition}

The proof of items 1. to 3.  can be found in \cite{JPh}, Proposition 2. The items 4. and 5. in \cite{DaSilva1} Proposition 2.13.

\begin{remark}
We notice that the item 4. of the above proposition together with the fact that $0\in\inner\Omega$ shows us in particular that $\AC$ is invariant by $\varphi_t$ for any $t\geq 0$. 
\end{remark}

\section{Controllability of linear systems}

In this section we extend the results obtained in \cite{DaSilva1} about controllability of linear systems on solvable  Lie groups to Lie groups with the following property:

\begin{definition}
\label{finitess}
Let $G$ be a connected Lie group. We say that the Lie group $G$ has finite semisimple center if all semisimple Lie subgroups of $G$ have finite center.
\end{definition}

\begin{remark}
We claim that there are many groups satisfying Definition \ref{finitess}. For instance, any solvable Lie group, any semisimple Lie group with finite center and any direct or semidirect product of these classes have semisimple finite center. Moreover, by Malcev's Theorem (see \cite{ALEB} Theorem 4.3) and its corollaries, any connected Lie group $G$ with one Levi subgroup $L$ with finite center also has semisimple finite center. 
\end{remark}

\bigskip
Let us denote by $G_0$, $G^+$, $G^-$, $G^0$, $G^{+, 0}$ and $G^{-, 0}$ the connected Lie subgroups of $G$ with Lie algebras $\fg_0$, $\fg^+$, $\fg^-$, $\fg^0$, $\fg^{+, 0}:=\fg^+\oplus\fg^0$ and $\fg^{-, 0}:=\fg^-\oplus\fg^0$, respectively. 

Next we state the main properties of the above subgroups. Its proof can be found in \cite{DaSilva1} Proposition 2.9.

\begin{proposition}
\label{subgroups}
It holds :
\begin{itemize}
\item[1.] $G^{+, 0}=G^+G^0=G^0G^+$ and $G^{-, 0}=G^-G^0=G^0G^-$;
\item[2.] $G^+\cap G^-=G^{+, 0}\cap G^-=G^{-, 0}\cap G^+=\{e\}$;
\item[3.] $G^{+, 0}\cap G^{-, 0}=G^0$;
\item[4.] All the above subgroups are closed in $G$;
\item[5.] If $G$ is solvable then 
$$G=G^{+, 0}G^-=G^{-, 0}G^+$$
Moreover, the fixed points of $\XC$ are in $G^0$;
\end{itemize}
\end{proposition}

The next results from \cite{DaSilva1} relate $\varphi$-invariant subsets of $G$ with the reachable set. 

\begin{lemma}
\label{pointinvariance}
Let $g\in\AC$ and assume that $\varphi_t(g)\in\AC$ for any $t\in\R$. Then $\AC\, g\subset\AC$.
\end{lemma}

\begin{proposition}
\label{idealinvariance}
Let $\fh$ be a Lie subalgebra of $\fg$ and $\fn$ an ideal of $\fh$ such that $\DC(\fh)\subset\fn$. If $N\subset\AC$ then $H\subset\AC$, where $H$ and $N$ are the connected Lie subgroups of $G$ with Lie subalgebras $\fh$ and $\fn$ respectively.
\end{proposition}

\begin{proposition}
\label{solvablesubgroups}
If $K\subset G^0$ is a connected $\varphi$-invariant solvable Lie subgroup of $G^0$ and $\AC$ is open then $G^0\subset\AC$.
\end{proposition}

Let $\alpha\neq 0$ be an eigenvalue of $\DC$ such that $\mathrm{Re}(\alpha)=0$ and consider the subspace of $\fg^0$ given by the sum of the multiples of $\alpha$ as
$$\fg^{\alpha}:=\bigoplus_{j\in\N}\fg_{j\alpha}.$$
Since $[\fg_{\alpha}, \fg_{\beta}]\subset\fg_{\alpha+\beta}$ when $\alpha+\beta$ is an eigenvalue of $\DC$ and zero otherwise, the subspace $\fg^{\alpha}$ is a nilpotent Lie subalgebra. If $\AC$ is open, we have by Proposition \ref{solvablesubgroups} that the connected subgroup of $G$ associated with the above Lie algebra is contained in $\AC$. In particular, $\exp (\fg_{\alpha})\subset\AC$ for any eigenvalue $\alpha\neq 0$ with zero real part.

\begin{remark}
The purpose of the paper is to extend the controllability results obtained in \cite{DaSilva1}, for solvable Lie groups. Since we proved that $\exp (\fg_{\alpha})\subset\AC$ for any eigenvalue $\alpha\neq 0$ of $\DC$ such that $\mathrm{Re}(\alpha)=0$, therefore we need to show that $G_0=\langle\exp(\fg_0)\rangle\subset\AC$. From Proposition \ref{solvablesubgroups} any solvable $\varphi$-invariant subgroup of $G_0$ is contained in $\AC$ we just need to take care of the semisimple component of $G_0$. In order to achieve our goal we construct a special semigroup and use Lemma 4.1 of \cite{SM3} that states: Any subsemigroup of a noncompact semisimple Lie group with finite center that contains a nilpotent element in its interior has to be the whole group.
\end{remark}

Let us assume that $e\in\inner\AC_{\tau_0}$ for some $\tau_0>0$. In the sequel we concentrate the study on Lie groups with finite semisimple center. Essentially, we generalize Proposition 3.5 above and Theorem 3.7 in \cite{DaSilva1}.

\bigskip

Since $\fg_0$ is $\DC$-invariant we can consider $\DC$ as a derivation of $\fg_0$. Let $\fr$ denote the solvable radical of the subalgebra $\fg_0$ and $R$ its associated connected solvable subgroup. Since the solvable radical is invariant under automorphisms we have that $\rme^{\tau\DC}\fr=\fr$ for any $\tau\in\R$ which implies that $\fr$ is $\DC$-invariant. Thus, we get a well defined derivation $\DC_*$ on the semisimple Lie algebra $\fl:=\fg_0/\fr$ with a commuting property, i.e., $\DC_*\circ\pi_*=\pi_*\circ\DC$, where $\pi_*:\fg_0\rightarrow \fl$ is the canonical projection. Moreover, since any derivation in a semisimple Lie algebra is inner there is $Y\in\fg_0$ such that $\DC_*=\ad_{\fl}(\pi_*(Y))$. 

By fixing $Y\in\fg_0$ above we have that $\DC^n(Z)-\ad(Y)^n(Z)\in\fr$ for any $n\in\N$ and $Z\in\fg_0$. Consequently, for any $\tau\in\R$ we have
\begin{equation}
\label{Rideal1}
\rme^{\tau\DC}Z=\rme^{\tau\ad(Y)Z}+W\;\;\;\mbox{ for some }\;\; W=W_{\tau, Y, Z}\in\fr.
\end{equation}
Consider the one parameter group of automorphisms of $G_0$ defined by 
$$\varphi_{\tau}^Y(h):=\exp(\tau Y)h\exp(-\tau Y).$$
Then, for any $Z\in\fg_0$ 
$$\varphi_{\tau}(\exp Z)=\exp(\rme^{\tau\DC}Z)\;\; \mbox{ and }\;\;\varphi_{\tau}^Y(\exp Z)=\exp(\rme^{\tau\ad(Y)}Z).$$ 
From equation (\ref{Rideal1}) and Lemma 3.1 of \cite{WM} we get 
$$\varphi_{\tau}(\exp Z)=\exp\left(\rme^{\tau\ad(Y)Z}+W\right)=\exp\left(\rme^{\tau\ad(Y)Z}\right)g=\varphi^Y_{\tau}(\exp Z)g,$$
where $g=g_{\tau, Y, Z}\in R$. 

Assume $h\in G_0$ of the form $h=h'\exp Z$ such that $\varphi_{\tau}(h')=\varphi_{\tau}^Y(h')g_1$ for some $g_1\in R$. It follows that $\varphi_{\tau}(\exp Z)=\varphi_{\tau}^Y(\exp Z)g_2$ for some $g_2\in R$. Therefore
$$\varphi_{\tau}(h)=\varphi_{\tau}(h'\exp Z)=\varphi_{\tau}(h')\varphi_{\tau}(\exp Z)=\varphi_{\tau}^Y(h')g_1\varphi_{\tau}^Y(\exp Z)g_2$$
$$=\varphi_{\tau}^Y(h')\varphi_{\tau}^Y(\exp Z)g=\varphi_{\tau}^Y(h'\exp Z)g=\varphi_{\tau}^Y(h)g$$
where $g=(\varphi_{\tau}^Y(\exp Z))^{-1}g_1\varphi_{\tau}^Y(\exp Z)g_2\in R$, because $R$ a normal subgroup of $G_0$.

Since $G_0$ is connected the above relation gives us by induction that 
\begin{equation}
\label{Rideal2}
\varphi_{\tau}(h)=\varphi_{\tau}^Y(h)g,\;\;\;\mbox{ for any } h\in G_0
\end{equation}
where  $g=g_{\tau, Y, h}\in R$.

Equation (\ref{Rideal2}) implies in particular that for any $\tau, \tau'\in\R$ it holds that 
$$\varphi_{\tau}(\exp(\tau' Y))=\exp(\tau'Y)g \;\;\mbox{ where } \;\;g=g_{\tau, \tau', Y}\in R.$$
The quocient $L:=G_0/R$ is a semisimple Lie group with Lie algebra $\fl$. Moreover, since we are assuming that $G$ has finite semisimple center, $L$ has finite center. Let us assume that $L$ is noncompact. By the previous analysis
$$\AC_{\tau}(\exp(\tau Y))=\AC_\tau\varphi_{\tau}(\exp(\tau Y))=\AC_{\tau}\exp(\tau Y)g, \;\;\;g\in R$$
consequently 
$$\pi(G_0\cap\AC_{\tau}(\exp(\tau Y)))=\pi((G_0\cap\AC_{\tau})\exp(\tau Y))$$ 
where $\pi:G_0\rightarrow L$ is the canonical projection. Consider for any $\tau>0$ the subsets of $L$ given by
$$\SC_{\tau}:=\pi((G_0\cap\AC_{\tau})\exp(\tau Y))\;\;\;\;\mbox{ and }\;\;\;\; \SC:=\bigcup_{\tau>0}\SC_{\tau}.$$

Since we are assuming $e\in\inner\AC_{\tau_0}$ for some $\tau_0>0$ we have that
$$\pi(\exp(\tau_0 Y))\in\pi((\inner\AC_{\tau_0}\cap G_0)\exp(\tau_0 Y))\subset\inner \SC$$ 
which implies that $\SC$ has nonempty interior in $L$. Next we show that $\SC$ is a subsemigroup of $L$, that is, $\SC$ is closed under group multiplication in $L$.

Let then $x_1, x_2\in\SC$. By definition, there are $\tau_i>0$, $u_i\in\UC$ such that 
$$x_i=\pi(\phi_{\tau_i, u_i}\exp(\tau_iY)) \;\;\mbox{ with }\;\; \phi_{\tau_i, u_i}\in\AC_{\tau_i}\cap G_0, \;\;\mbox{ for }\;\;i=1,2.$$ 
It follows that  
$$x_2x_1=\pi(\phi_{\tau_2, u_2}\exp(\tau_2Y)\phi_{\tau_1, u_1}\exp(\tau_1Y)).$$
Furthermore,
$$\phi_{\tau_1, u_1}(\exp((\tau_2+\tau_1)Y))=\phi_{\tau_1, u_1}\varphi_{\tau_1}(\exp((\tau_2+\tau_1)Y)$$
$$=\phi_{\tau_1, u_1}\exp((\tau_2+\tau_1)Y)g_1=\left(\phi_{\tau_1, u_1}\exp(\tau_1Y)\right)\exp(\tau_2Y)g_1$$
where $g_1\in R$. By considering the concatenation $u\in\UC$ between $u_1$ and $u_2$ we obtain 
$$\phi_{\tau_2+\tau_1, u}(\exp((\tau_2+\tau_1)Y))=\phi_{\tau_2, \Theta_{\tau_1}u}(\phi_{\tau_1, u}(\exp((\tau_2+\tau_1)Y)))$$
$$=\phi_{\tau_2, u_2}\varphi^Y_{\tau_2}(\phi_{\tau_1, u_1}(\exp((\tau_2+\tau_1)Y)))g_2=\phi_{\tau_2, u_2}\exp(\tau_2Y)\phi_{\tau_1, u_1}\exp(\tau_1Y)g$$
where $g=\varphi^Y_{\tau_2}(g_1)g_2\in R$. 

Therefore 
$$\phi_{\tau_2, u_2}\exp(\tau_2Y)\phi_{\tau_1, u_1}\exp(\tau_1Y)\in G_0\cap\AC_{\tau_2+\tau_1}(\exp(\tau_2+\tau_1)Y)g^{-1}$$
and since $G_0=G_0(\exp(\tau_2+\tau_1)Y)g^{-1}$ we get 
$$G_0\cap\AC_{\tau_2+\tau_1}(\exp(\tau_2+\tau_1)Y)g^{-1}=(G_0\cap\AC_{\tau_2+\tau_1})(\exp(\tau_2+\tau_1)Y)g^{-1}.$$
Finally, by taking the projection 
$$x_2x_1\in\pi((G_0\cap\AC_{\tau_2+\tau_1})\exp((\tau_2+\tau_1)Y))=\SC_{\tau_2+\tau_1}\subset\SC$$ 
showing that $\SC$ is a subsemigroup of $L$.

Since $\DC$ restricted to $\fg_0$ is nilpotent, by considering $X=\tau_0\,\pi_*(Y)$ we have that 
$$\exp_{L}(X)=\pi(\exp(\tau_0 Y))\in\inner\SC \;\;\;\mbox{ and }\;\;\;\ad_{\fl}(X)^j=0, \mbox{ for some }\; j\in\N.$$ 
By Lemma 4.1 of \cite{SM3} it follows that $\SC=L$.

\bigskip
We are now able to prove a generalization of Proposition \ref{solvablesubgroups} above.

\begin{proposition}
\label{generalcase}
Let $G$ be a connected Lie group with finite semisimple center. If $\AC$ is open, it follows that $G^0\subset\AC$.
\end{proposition} 

\begin{proof}
We first show that $G_0\subset\AC$. By considering our previous analysis we know that $R$ is a connected $\varphi$-invariant solvable Lie subgroup of $G^0$. Therefore Proposition \ref{solvablesubgroups} implies that $R\subset\AC$. Now, for the semisimple Lie group $L=G_0/R$ there are only two topological possibilities:
\begin{itemize}
\item[1.] $L$ is compact: In this case the Cartan-Killing form of $\fl$ is an inner product. Since any derivation is skew-simmetric with respect to this form and $\DC|_{\fg_0}$ is nilpotent, we must have that $\DC_*\equiv 0$ which implies that $\DC(\fg_0)\subset\fr$. 
Therefore, the result follows from Proposition \ref{idealinvariance}. 

\item[2.] $L$ is noncompact: For this case, we already proved that $\SC=L$. Consider the Lie subalgebra of $\fg_0$ defined by $\fh:=\R Y+\fr$. We notice that the sum is direct iff $Y\notin\fr$. Because of our choices $\DC(Y)\in\fr$. On the other hand $\fr$ is an ideal of $\fg_0$. It turns out that $\fh$ is a $\DC$-invariant Lie subalgebra of $\fg_0$ such that $\DC(\fh)\subset\fr$. Since $R\subset\AC$, Proposition \ref{idealinvariance} implies that $H\subset\AC$ where $H$ is the connected subgroup with Lie algebra $\fh$. Since $H$ is $\varphi$-invariant and is contained in $\AC$, Lemma \ref{pointinvariance} implies $\AC \,g\subset\AC$ for any $g\in H$ and consequently
$$\AC_{\tau}\exp(\tau Y)\subset\AC\exp(\tau Y)\subset\AC, \;\;\;\tau>0.$$
In particular, $\SC\subset\pi(\AC\cap G_0).$ Then, by using Lemma \ref{pointinvariance} again, we obtain
$$G_0=(\AC\cap G_0)R\subset\AC R\subset \AC.$$
\end{itemize}
So, in any case, $G_0\subset\AC$. Since $\exp(\fg_{\alpha})\subset\AC$ for any $\alpha\neq 0$ with $\mathrm{Re}(\alpha)=0$, Lemma \ref{pointinvariance} shows that 
$$B:=\prod_{\alpha\;: \mathrm{Re}(\alpha)=0}\exp\fg_{\alpha}\subset\AC.$$
Since $B$ is a $\varphi$-invariant neighborhood of the identity element in $G^0$ we obtain by Corollary 3.3 of \cite{DaSilva1} that $G^0\subset\AC$ concluding the proof.
\end{proof}

The next result generalizes Theorem 3.7 of \cite{DaSilva1}.

\begin{theorem}
\label{generalcase2}
Let $G$ be a connected Lie group with finite semisimple center. If $e\in\inner\AC_{\tau_0}$ for some $\tau_0>0$, then $G^{+, 0}\subset\AC$.
\end{theorem}

\begin{proof}
Let $g\in G^+$. Since $\rme^{t\DC}|_{\fg^+}$ has only eigenvalues with positive real part and $\AC$ is open, there is $t>0$ great enough such that $\varphi_{-t}(g)\in\AC$ and consequently $g\in\varphi_{t}(\AC)\subset\AC$. Since $g\in G^+$ is arbitrary, we conclude that $G^+\subset\AC$. Moreover, by Proposition \ref{generalcase} $G^0\subset\AC$ and since $G^{+, 0}=G^+G^0$, Lemma \ref{pointinvariance} implies that $G^{+, 0}\subset\AC$ as desired.
\end{proof}

\bigskip

The proof of our main result uses the notion of reverse system. So, next we show some relationship between both systems: the linear one and its linear reverse.

By considering the linear vector field $\XC^*$ on $G$ whose flow is given by $\varphi_{\tau}^*:=\varphi_{-\tau}$ it is straightforward to see that the derivation $\DC^*$ associated with $\XC^*$ satisfies $\DC^*=-\DC$. The Lie subalgebras and Lie subgroups induced by the derivation $\DC^*$ are related with the ones induced by $\DC$ as
$$\fg_*^+=\fg^-, \;\;\;\;\;\fg^-_*=\fg^+,\;\;\; \mbox{ and }\;\;\;\fg^0_*=\fg^0$$
and
$$G_*^+=G^-, \;\;\;\;\;G^-_*=G^+,\;\;\; \mbox{ and }\;\;\;G^0_*=G^0.$$
Moreover, if we consider the linear system (\ref{linearsystem}) with drifts $\XC^*$, $\XC$ and the same right invariant vector fields, their respectives solutions are related by $\phi^*_{t, u}(g)=\phi_{-t, u}(g)$ which implies that $\AC^*_{\tau}=\varphi_{-\tau}((\AC_{\tau})^{-1})$. 

\begin{remark}
It is straighforward to see that the linear system (\ref{linearsystem}) is controllable if and only if $G=\AC\cap\AC^*$.
\end{remark}

With the previous analysis we get the main result of the paper.

\begin{theorem}
\label{controllability}
Let $G$ be a Lie group with finite semisimple center. Then, the linear system (\ref{linearsystem}) on $G$ is controllable if $e\in\inner\AC_{\tau_0}$ for some $\tau_0>0$ and $\DC$ has only eigenvalues with zero real part.
\end{theorem}

\begin{proof}
Since $\AC^*_{\tau}=\varphi_{-\tau}((\AC_{\tau})^{-1})$ for any $\tau>0$ we have that $e\in\inner\AC_{\tau_0}$ for some $\tau_0>0$ if and only if $e\in\inner\AC^*_{\tau_0}$. By Theorem \ref{generalcase2} it follows that $G^{+, 0}\subset\AC$ and $G^{-, 0}=G^{+, 0}_*\subset\AC^*$. Since $\DC$ has only eigenvalues with zero real part, $G^{+, 0}=G^0=G^{-, 0}$ and consequently $G=G^0=\AC\cap\AC^*$ which implies that (\ref{linearsystem}) is controllable.
\end{proof}

\begin{remark}
For restricted systems, where $\Omega$ is a compact, convex subset of $\R^m$, Lemma 4.5.2 of \cite{FCWK} implies that the condition $e\in\inner\AC_{\tau_0}$ for some $\tau_0>0$ is equivalent to the openness of the reachable set. 
\end{remark}

\begin{example}
The classical linear system: consider $G=\R^d$ and the dynamic on $G$ determined by
$$\dot{x}(t)=Ax(t)+Bu(t); \;\;A\in\mathbb{R}^{d\times d}, B\in\mathbb{R}^{d\times m}, u\in\UC.$$
Since the right (left) invariant vector fields on the abelian Lie group $\R^d$ are given by constant vectors we can write the system as 
$$\dot{x}(t)=Ax(t)+\sum_{i=1}^mu_i(t)b_i, \hspace{.5cm} B=(b_1| b_2| \cdots|b_m)$$
showing that it is a linear system in the sense of (\ref{linearsystem}).  Of course $\partial\R^d=\mathrm{gl}(n, \R)$, the Lie algebra of the real matrices of order $d$.

In the book The Dynamics of Control \cite{FCWK} the authors proved that for a restricted linear system on the Euclidean space $\R^d$ satisfying the Kalman condition, there exists one and only one control set C with non empty interior. Recall that a control set $C$ is a maximal controlled invariant set such that $C\subset \cl(\AC(g))$ for every $g\in C$. The mentioned control set is given explicitely by
$$C = \cl(\AC) \cap \AC^*.$$

If we assume that $A$ has only eigenvalues with zero real part, then $\R^d = (\R^d)^0 = \AC\cap\AC^*$. It turns out that $C = \R^d$ and the system is controllable. Therefore, our main results is also a generalization for restricted linear systems from Euclidean spaces to Lie groups with finite semisimple center.
\end{example}

\begin{example}
A special class of linear control systems: Consider two Lie algebras $\fe$ and $\fh$ with respective connected Lie groups $E$ and $H$, and the direct product $\fg = \fe\times\fh$ with the canonical product Lie algebra structure. Let $X^{\fe, 1}, \ldots, X^{\fe, m}$ be right invariant vector fields of $\fe$ and $X^{\fh, 1}, \ldots, X^{\fh, n}$ be a basis of right invariant vector fields of $\fh$. Take any derivation $\DC^{\fe}$ of $\fe$ with associated linear vector field $\XC^{\fe}$. Therefore
$$\dot{g}(t) = X^{\fe}(e(t))+\sum_{j=1}^mu_j(t)X^{\fe, j}(e(t))+\sum_{j=1}^nv_j(t)X^{\fh, j}(h(t)),$$
is a linear control system that we call a $\fh$-homogeneous linear system on $G=E\times H$. Here $g(t) = (e(t),h(t))$ for any $t\in\R $ and $w = (u,v)$ belongs to $\UC \subset L^{\infty}(\R,\Omega\subset \R^m \times \R^n)$. Just observe that $\DC=(\DC^{\fe},0)$ is a derivation of $\fg$ and $\fg^0=\fg$.

\begin{remark} Obviously, a $\fh$-homogeneous linear system on $G$ is controllable if and only if the linear system is controllable when $v = 0$.
\end{remark}
In Jurdjevic's book \cite{VJ} the kinematic equations of a rolling two dimensional sphere $S$ over a revolving plane around the $z$-axis with constant angular velocity $\omega$, are stablished as follows: The assumption that the sphere roll without slipping implies that the movement is described by the center of $S$ as a curve $e(t)$ on $\R^2$ and by the family of orthogonal matrices $g(t)$, i.e., elements of the orthogonal group $SO(3)$, which transform the attached frame on $S(t)$ to the canonical coordinates on $R^3$ at the origin. So, the configuration space is the Lie group $G = \R^2\times \mathrm{SO}(3)$. Then, $\fe = \R^2$, $\fh = \fs\fo(3)$ and $\fg = \R^2\times \fs\fo(3)$ where the second component denotes the Lie algebra of skew-symmetric matrix of order three. After this analysis, it is obtained the linear system
$$\dot{g}(t) = X^{\fe}(e(t))+\sum_{j=1}^3v_j(t)X^{\fh, j}(h(t)),$$
a $\fs\fo(3)$-homogeneous linear system in our context. Precisely
$$
X^{\fh, 1}=\left(
\begin{array}{ccc}
0 & 0 & 0\\
0 & 0 & -1\\
0 & 1 & 0
\end{array}\right),\;
X^{\fh, 2}=\left(
\begin{array}{ccc}
0 & 0 & -1\\
0 & 0 & 0\\
1 & 0 & 0
\end{array}\right),\;
X^{\fh, 3}=\left(
\begin{array}{ccc}
0 & -1 & 0\\
1 & 0 & 0\\
0 & 0 & 0
\end{array}\right).
$$
and  
$$
\DC^{\fe}=\omega\left(
\begin{array}{cc}
0 & -1\\
1 & 0
\end{array}\right).$$
Of course $\DC=(\DC^{\fe},0)$ does not act on $\fh$. So, in order to study the controllability property of the rolling sphere it is necessary to consider an extension of the system: a $\fs\fo(3)$-homogeneous linear system on $G$ as follows
\begin{equation}
\label{Jurdjevic}
\dot{g}(t) =\XC^{\fe}(e(t)) + ub +\sum_{j=1}^3v_j(t)X^{\fh, j}(h(t)), \mbox{ with } \Omega= \R \times\R^3
\end{equation}

where $b\in\R^2$ is a vector such that the system satisfy the Kalman rank condition for $v = 0$. For example if $b = e_1$ we obtain 
$$\R^2=\mathrm{Span}\{b, \DC^{\fe}(b)\}.$$
The Lie group $G$ has finite semisimple center, the system satisfy the ad-rank condition and $\DC$ is a $\fg$-derivation with $\mathrm{Spec}(\DC) = \{0, \pm\omega i\}$. Thus, our main theorem applies to the $\fs\fo(3)$-homogeneous linear system showing that (\ref{Jurdjevic}) it is controllable.

In fact, at the origin and with $u = 0$ it is possible to obtain any configuration of $S(0)$ by using the controls $v_1,v_2$ and $v_3$. On the other hand, with $v = 0$ it is possible to connect any two elements of $\R^2$ in positive time.
\end{example}

\begin{example}
Let $G=SL(2,\mathbb{R)}$ the Lie group of determinant $1$ matrices with Lie
algebra $\mathfrak{sl}(2,\mathbb{R})$ the matrices of $0$ trace and consider the basis $\{X, Y, Z\}$ of $\mathfrak{sl}(2, \R)$ given by
$$
X=\left(
\begin{array}{cc}
1 & 0 \\ 
0 & -1%
\end{array}\right), \;\;\;
Y=\left(
\begin{array}{cc}
0 & 0 \\ 
1 & 0%
\end{array}\right), \;\mbox{ and }\;
Z=\left(
\begin{array}{cc}
0 & 1 \\ 
0 & 0%
\end{array}\right)
$$
that satisfies
$$[X, Y]=-2Y, \;\;\;[X, Z]=2Z \;\;\;\mbox{ and }\;\;\;[Z, Y]=X.$$ 

\begin{itemize}

\item[i)] In \cite{VASM} the authors show the existence of a local controllable
linear control system from the identity on $G$ such that the accessibility
set $\mathcal{A}$ is not a semigroup and then is not controllable (see Proposition 7 of \cite{JPh}). Precisely, let 
$$
\DC=\ad(X)\;\;\text{ and }\;\;H=\left(
\begin{array}{cc}
1 & 1 \\ 
1 & -1
\end{array}\right).
$$
and consider the linear system 
$$
\dot{g}(t)=\mathcal{X}(g(t))+uH(g(t)), \;\;\;u\in\UC, 
$$
where $\XC=\XC^{\DC}$ is the linear vector field associated with $\DC$.
We have 
$$
\mathrm{Span}\left \{H=X+Y+Z,\DC(H)=2(X-Y),\DC^{2}(H)=4Y\right\}=\mathfrak{sl}(2,\mathbb{R)} 
$$
and the rank condition follows. On the other hand, the center of $\mathrm{SL}(2,\mathbb{R})=\mathbb{Z}_{2}$ and we observe that $\mathrm{Spec}(\mathcal{D})=\{0, \pm 2\}$.

\item[ii)] Consider 

The derivation 
$$\DC=\ad(Y)=\left(\begin{array}{ccc}
0 & 0 & -1\\
2 & 0 & 0 \\
0 & 0& 0
\end{array}\right)$$
induces the linear vector field $\XC=\XC^{\DC}$ on $G$. Consider the linear system
$$\dot{g}(t)=\XC(g(t))+uZ(g), \;\;\;u\in\UC.$$
Since 
$$\mathrm{Span}\{Z, \DC(Z)=-X, \DC^2(Z)=-2Y\}=\mathfrak{sl}(2, \R)$$
and $\mathrm{Spec}(\DC)=0$ we have by Theorem \ref{controllability} that the linear system is controllable.

\end{itemize}
\end{example}

\end{document}